\documentclass[fontsize=12pt]{amsart}

\usepackage[object=vectorian]{pgfornament}
\usepackage{amsmath}
\usepackage{amsthm}
\usepackage{amssymb}
\usepackage{mathtools}
\usepackage{mathrsfs}
\usepackage{physics}
\usepackage{scrextend}
\usepackage{tikz-cd}
\usepackage[shortlabels]{enumitem}
\usepackage{colonequals}
\usepackage[hyperfootnotes]{hyperref}
\usepackage{comment}
\usepackage[margin=1in]{geometry}
\usepackage{caption}

\usepackage[notcolon, notquote]{hanging}

\renewcommand{\P}{\mathbb{P}}

\newcommand{\C}{\mathbb{C}}
\newcommand{\cC}{\mathcal{C}}
\newcommand{\Z}{\mathbb{Z}}
\newcommand{\Q}{\mathbb{Q}}

\newcommand{\I}{\mathcal{I}}
\renewcommand{\O}{\mathcal{O}}
\newcommand{\X}{\mathcal{X}}

\DeclareMathOperator{\Jac}{Jac}

\DeclareMathOperator{\Bl}{Bl}
\DeclareMathOperator{\Pic}{Pic}

\DeclareMathOperator{\gon}{gon}
\DeclareMathOperator{\covgon}{cov.gon}
\DeclareMathOperator{\conngon}{conn.gon}
\DeclareMathOperator{\irr}{irr}
\DeclareMathOperator{\uniirr}{uni.irr}
\DeclareMathOperator{\stabirr}{stab.irr}
\DeclareMathOperator{\corrdeg}{corr.deg}

\DeclareMathOperator{\Gr}{Gr}
\DeclareMathOperator{\Bir}{Bir}
\DeclareMathOperator{\Aut}{Aut}
\DeclareMathOperator{\Sym}{Sym}

\newcommand{\mybigwedge}{\raisebox{.25ex}{\scalebox{0.86}{$\bigwedge$}}}

\newtheorem{thm}{Theorem}[section]
\newtheorem{corollary}[thm]{Corollary}
\newtheorem{lemma}[thm]{Lemma}
\newtheorem{proposition}[thm]{Proposition}

\theoremstyle{definition}
\newtheorem{conjecture}[thm]{Conjecture}
\newtheorem{definition}[thm]{Definition}

\theoremstyle{remark}
\newtheorem{remark}[thm]{Remark}
\newtheorem{question}[thm]{Question}
\newtheorem{problem}[thm]{Problem}
\newtheorem{example}[thm]{Example}

\begin{document}

\title{A primer on measures of irrationality}
\author{Nathan Chen and Olivier Martin}

\allowdisplaybreaks

\thispagestyle{empty}

\maketitle

\tableofcontents

\section*{Introduction}

Rationality problems have been at the forefront of algebraic geometry for at least 150 years. The most classical of such problems is to determine which varieties are birational to projective space. Historically, in dimensions at least two, the subject began with the L\"{u}roth problem for surfaces, and the study of coarse invariants such as the space of holomorphic forms and the space of $m$-canonical global sections. This later paved the way for the Kodaira--Enriques classification of algebraic surfaces. Although questions concerning rationality and the behavior of these invariants in higher dimension are extremely subtle, there has been a flurry of activity in recent decades (see \cite{SurveyRationalityProblems, Debarre24} for a survey).

However, since most varieties are not rational, it is natural to ask how far a non-rational variety is from being rational. More generally one can even ask how birationally proximate two algebraic varieties are. Measures of irrationality are numerical invariants of algebraic varieties which quantify how far they are from being rational (or rationally connected, uniruled, etc.). The study of these invariants is a new area of research which is still in an exploratory phase. For lack of broadly applicable results, the current focus has been towards building a collection of examples and developing machinery to either obstruct the existence of certain rational maps or construct said maps. Some of the techniques have been inspired by obstructions to rationality, but often times new ideas are required to deal with higher degree maps. The purpose of this survey is to introduce and motivate measures of irrationality, summarize the current literature, and compile a list of open problems and directions for future research.

Recall that the \textit{gonality} of a smooth projective curve $C$, denoted $\gon(C)$, is defined as the minimal degree of a branched cover $C \longrightarrow \mathbb{P}^1$. One can extend this definition to integral curves $C$ by considering the minimal degree of a dominant rational map $C \dashrightarrow \P^{1}$. From the definition, we see that a smooth projective curve has gonality $1$ if and only if it is isomorphic to $\mathbb{P}^1$. Curves of genus $\geq 2$ and gonality $2$ are called hyperelliptic, and can be realized as double covers of $\mathbb{P}^1$ ramified over $r$ points, where $r>4$ is even. By Riemann--Hurwitz, the genus of such a curve is $g=(r-2)/2$. In particular, there exist hyperelliptic curves of arbitrarily large genus.

On the other hand, a curve $C$ of genus $g>0$ cannot have arbitrarily large gonality. Indeed, a well-known fact from Brill--Noether theory is that these invariants are related by the inequalities
\[ 2 \leq \gon(C) \leq \left\lfloor \frac{g+3}{2} \right\rfloor, \]
and every possible value in this range is achieved by some curve of genus $g$. In flat families of integral curves, gonality is a lower semicontinuous function in the Zariski topology. Inside the coarse moduli space $M_{g}$, curves of gonality $\leq k$ (for any integer $k\geq 2$) form an irreducible subvariety of dimension equal to $\text{min}(2g+2k-5, 3g-3)$. We recommend \cite{ACGH85} for a thorough treatment of gonality through the lens of Brill--Noether theory.

\begin{example}[M. Noether]\label{ex:noe}
A theorem of Max Noether (see \cite{MN} for the original source and \cite{Ci,Ha} for complete proofs) states that any smooth plane curve
\[ C \subset \P^{2} \quad \text{of degree} \quad d\geq 2 \]
has gonality $d-1$. Moreover, for $d \geq 3$ any map $C \rightarrow \P^{1}$ of degree $d-1$ is given by projection from a point. In light of the previous discussion on moduli, this illustrates the fact that plane curves are quite special since their gonality grows linearly in $d$, whereas their genus grows quadratically in $d$.
\end{example}

In higher dimensions, it is natural to ask for a suitable analogue of gonality and two main generalizations have been studied:

\begin{definition}
The \textit{degree of irrationality} of $X$ is 
\[ \irr (X) \coloneqq \min\{\deg \varphi \mid \exists \ \varphi: X\dashrightarrow \mathbb{P}^{\dim X} \text{ dominant} \}. \]
\end{definition}

\begin{definition}
The \textit{covering gonality} of $X$ is
\[ \covgon(X) \coloneqq \min\begin{cases} c \in \mathbb{Z}_{\geq 0}\;\Bigg|\begin{rcases} \text{ a general point } x\in X \text{ is contained }\\\text{ in an integral curve of gonality } c \end{rcases}.\end{cases} \]
\end{definition}

\begin{remark}
Over an arbitrary field $k$, the degree of irrationality can be defined algebraically as the smallest positive integer $d$ such that there exists a transcendence basis $\alpha_1,\ldots,\alpha_n$ of $k(X)/k$ satisfying $[k(X):k(\alpha_1,\dots, \alpha_n)]=d$. The covering gonality can be reinterpreted as the smallest positive integer $c$ such that there exists a diagram:
\begin{center}
\begin{equation}\label{diagcovgon}
\begin{tikzcd}
\cC \arrow[d, swap, "\pi"] \arrow[r, dashed, "f"] & X \\
B, & 
\end{tikzcd}
\end{equation}
\end{center}
where $\pi$ is a proper family of curves, the map $f$ is dominant and generically finite, and the general fiber of $\pi$ is a smooth curve of gonality equal to $c$
\end{remark}

These invariants only depend on $X$ up to birational equivalence. The degree of irrationality first appeared in a paper of Heinzer and Moh \cite{HM82}, in which the authors studied the degrees of extensions of function fields of varieties. The covering gonality seems to have been first studied by Lopez and Pirola \cite{LP} for very general surfaces $S \subset \P^{3}$.

Let us mention a few properties of these invariants:
\vspace{-1em}

\begin{enumerate}[wide, labelwidth=!, labelindent=0pt, label=(\textbf{\alph*})]
\item When $X$ is a curve, both the covering gonality and the degree of irrationality coincide with $\gon(X)$. Moreover, $\text{irr}(X)=1$ (resp. $\covgon(X) = 1$) if and only if $X$ is rational (resp. covered by rational curves, which is to say that $X$ is \textit{uniruled}).
\item Since a covering family of rational curves on $\mathbb{P}^{\dim X}$ can be pulled back under a dominant rational map $X\dashrightarrow \mathbb{P}^{\dim X}$ of degree $d$ to obtain a (birational) covering of $X$ by $d$-gonal curves, we obtain the inequality:
\begin{equation*}
\covgon(X) \leq \irr(X).
\end{equation*}
By Noether normalization we see that both invariants are finite.
\item If $Y\dashrightarrow X$ is a dominant generically finite rational map of degree $d$ then
\[ \text{irr}(Y)\leq d\cdot \text{irr}(X),\text{\qquad and \qquad}\text{cov.gon}(X)\leq \text{cov.gon}(Y)\leq d\cdot \text{cov.gon}(X). \]
The second equality uses the fact that curves $C,D$ satisfy $\gon(D)\leq \gon(C)$ if there is a dominant rational map $C\dashrightarrow D$.
\end{enumerate}

\begin{example}\label{ex:productcurves}
Let $C_1,C_2$ and be smooth projective curves. Then
$$\text{cov.gon}(C_1\times C_2)=\text{min}(\gon(C_1),\gon(C_2)).$$
On the other hand, the degree of irrationality of $C_1\times C_2$ is mysterious in general, despite the obvious upper bound coming from a product of maps to $\mathbb{P}^1\times \mathbb{P}^1$:
$$\text{irr}(C_1\times C_2)\leq \gon(C_1)\cdot \gon(C_2).$$
Given a dominant rational map of degree $d$
$$\varphi: C_1\times C_2\dashrightarrow \mathbb{P}^2,$$
the Zariski closure of the preimage of a general line $\ell \subset \mathbb{P}^2$ is an irreducible curve of gonality $d$ which dominates each factor under the projection maps. Since the gonality of a cover of a curve is bounded below by the gonality of the base curve, we deduce that
$$d=\gon(D)\geq \text{max}(\gon(C_1),\gon(C_2)),$$
and therefore
$$\max (\gon(C_1),\gon(C_2))\leq \text{irr}(C_1\times C_2)\leq \gon(C_1)\cdot \gon(C_2).$$
Of course this lower bound is very bad in general and the expectation is that if $C_1$ and $C_2$ are not rational and are general in moduli then the upper bound is attained. We will return to this question in \S\ref{prodsection}.
\end{example}

\begin{example}[Heinzer--Moh]\label{ex:HM}
If $f: X \dashrightarrow C$ is a dominant rational map from a variety to a curve, then $\irr(X) \geq \gon(C)$. Indeed, if $\varphi: X\dashrightarrow \mathbb{P}^n$ is generically finite of degree $d$ then we can find an irreducible rational curve $D\subset \mathbb{P}^n$ such that the strict transform $D'$ of $D$ is neither contained in the indeterminacy locus of $f$ nor a fiber of $f$. Hence $D'$ dominates $C$ and we deduce that
$$d\geq \text{gon}(D)=\text{gon}(D')\geq \gon(C).$$
\end{example}

\begin{remark} In contrast with the previous example, the existence of a dominant rational map $X \dashrightarrow Y$ to a higher dimensional variety $Y$ does not imply that $\irr(X) \geq \irr(Y)$. Indeed, the field of rationality problems provides many examples of varieties which are unirational, which is to say dominated by a projective space, yet not rational. For examples involving varieties with higher degrees of irrationality, one can take quotients of abelian surfaces or hyperelliptic surfaces (see e.g. Table~\ref{hyperirr} in \S\ref{kodsection} or \cite[Example 3]{Y2}).
\end{remark}

\begin{remark}\label{projection}
Let $X\subset \mathbb{P}^N$ be a non-degenerate subvariety of degree $d$ and dimension $n$. Projecting from the span of $(N-n)$ general points on $X$ gives a dominant rational map of degree $d - (N - n)$ to $\mathbb{P}^n$, so that $\irr(X) \leq d - (N - n)$. However, this approach is not expected to give good bounds in general. For example, it fails terribly for most curves embedded in $\P^{3}$ by very ample line bundles since the degree of the embedding will be too large. However, we will see later on that this bound is (near) optimal for very general hypersurfaces and complete intersections of large (multi)degree. One can obtain a slightly better bound on the covering gonality by taking highly tangent planes to $X$, but again these bounds are not expected to be good in general.
\end{remark}

\begin{example}\label{ex:IskMan}
A degree $2$ generically finite rational map $X\dashrightarrow \mathbb{P}^n$ gives rise to a birational involution of $X$. This makes showing that $\text{irr}(X)\geq 3$ for a large class of varieties quite easy since one can bring to bear the full power of modern birational geometry. For example, Iskovskikh--Manin \cite{IM71} prove that a smooth quartic threefold $Y$ is birationally rigid. In particular, $\text{Bir}(Y)=\text{Aut}(Y)$ and $\irr(Y) \not= 1$. If $Y$ is very general, then $\text{Aut}(Y)$ is trivial so $\text{irr}(Y)\geq 3$. Projecting from a point on $Y$ gives a degree $3$ rational map to $\mathbb{P}^4$, showing that $\text{irr}(Y)=3$.

\end{example}

\begin{example}\label{ex:maxalb}
Let $X$ be a smooth projective surface of maximal Albanese dimension. We will show that $\text{irr}(X)\geq 3$. Since $X$ has maximal Albanese dimension there are $1$-forms $\eta,\eta'\in H^{0}(X,\Omega^1_X)$ such that $\eta\wedge \eta'\neq 0$. Since $h^{p,0}(X) \coloneqq h^0(X,\Omega_X^p)$ is a birational invariant of smooth projective varieties, $X$ is not rational. Moreover, if there was a rational double cover $f: X\dashrightarrow \mathbb{P}^2$ with covering involution $\tau: X\dashrightarrow X$, then $\tau$ would act as $(-1)$ on the group $H^{0}(X,\Omega^1_X)$, therefore leaving $\eta\wedge \eta'$ invariant. It follows that this holomorphic $2$-form would descend to a non-zero section of $K_{\mathbb{P}^2}$, providing a contradiction. This example was generalized by Alzati--Pirola in \cite{AP2} (see Theorem \ref{APthm}).
\end{example}

It has proven quite difficult to obtain generally applicable lower bounds beyond $3$ for the degree of irrationality. From an algebraic perspective, the difficulty arises from the fact that extensions of degree $>2$ need not be Galois, so good knowledge of the birational automorphism group of $X$ is of little use. For instance, the authors do not know how to show that any K3 surface has degree of irrationality strictly larger than $3$.

One of the goals of this survey is to collect some of the (admittedly somewhat ad hoc) techniques which have been used to obtain bounds on measures of irrationality. One source of inspiration for techniques in the field of measures of irrationality are rationality problems. Indeed, obstructions to rationality can sometimes be generalized to obtain obstructions to the existence of low degree rational maps from a variety to projective space. For instance, if $X$ is a smooth projective variety such that $H^0(X,K_X)\neq 0$, then $X$ is not rational; we will see in \S\ref{postech} that if $K_X$ is suitably positive then $X$ cannot admit a generically finite rational map of low degree to a projective space. Similarly, a technique of Koll\'{a}r which was used to show irrationality of some Fano hypersurfaces has been adapted to the setting of the degree of irrationality by the first author and Stapleton (see \S 2.3 for details). It is therefore natural to ask which other obstructions to rationality can be adapted to the setting of measures of irrationality. In upcoming work, the second author will present an obstruction generalizing the decomposition of the diagonal together with applications, which though modest, unify some results in the literature. On the other hand, other techniques, such the Clemens--Griffiths method, do not seem to have natural generalizations.

In the first section, we will summarize the current state of knowledge on measures of irrationality for surfaces, following the Enriques--Kodaira classification. In the second section, we discuss positivity techniques and their application to measures of irrationality of hypersurfaces and complete intersections in projective space. In the third section, we address measures of irrationality of abelian and irregular varieties. In the last section, we conclude with some complements, conjectures, and open questions. We will mostly work over the complex numbers unless otherwise stated.

\vspace{0.5cm}

\section{Algebraic surfaces}

A first step towards understanding measures of irrationality is to determine what values these invariants can take for algebraic surfaces. In what follows, $S$ is a smooth projective surface and we review what is known about measures of irrationality for the different classes of the Enriques--Kodaira classification.

\subsection{Kodaira dimension $-\infty$}

\begin{itemize}[wide, labelwidth=!, labelindent=0pt, label=$\diamond$]
\item \textbf{Rational surfaces:} Both the covering gonality and the degree of irrationality are equal to $1$.
\item \noindent\textbf{Ruled surfaces:} $S$ is birational to $C\times \mathbb{P}^1$ for some curve $C$ of positive genus. Thus, from Example~\ref{ex:HM} we see that
\[ 1=\textup{cov.gon}(S)< \textup{irr}(S)=\textup{gon}(C). \]
\end{itemize}

\subsection{Kodaira dimension $0$}
\label{kodsection}
\begin{itemize}[wide, labelwidth=!, labelindent=0pt, label=$\diamond$] 
\item \textbf{K3 surfaces:} A theorem of Bogomolov and Mumford \cite[p. 351]{MM} states that a K3 surface $S$ is covered by (singular) elliptic curves. Since $S$ is not uniruled it follows that
\[ \textup{cov.gon}(S)=2. \]
To illustrate how little is known about the degree of irrationality of K3 surfaces, the authors do not know of a single example of a K3 surface with $\irr(S) \geq 4$.

In the remainder of this subsection, we will use $(S_d, L_d)$ to denote a very general polarized $K3$ surface of degree $d$. The main open question regarding measures of irrationality for K3 surfaces is the following:

\begin{conjecture}[Conjecture 4.2 in \cite{BDPELU}]\label{K3conj}
$$\limsup_{d\to\infty}\textup{irr}(S_d)=\infty.$$
\end{conjecture}

\noindent In the opposite direction, Stapleton shows the following result in his thesis:
\begin{thm}[\cite{Stap} Theorem 5.1]
There is a constant $C$ such that
$$\textup{irr}(S_d)\leq C\sqrt{d}.$$
\end{thm}
\noindent He also conjectures that this upper bound gives the correct asymptotic:
\begin{conjecture}[\cite{Stap}]
    Let $(S_{d}, L_{d})$ denote a very general polarized $K3$ surface of degree $d$. Then there exist constants $C_{1}, C_{2} > 0$ such that
    \[ C_{1} \sqrt{d} \leq \irr(S_{d}) \leq C_{2} \sqrt{d} \]
\end{conjecture}

As a first step towards these questions, one can ask:

\begin{question}
Is there an infinite subset $I\subset 2\mathbb{Z}_{>0}$ and a constant $M>0$ such that $\text{irr}(S_d)\leq M$ for all $d\in I$? As pointed out by Ottem and Gounelas, this is true if one replaces $\text{irr}(S_d)$ by $\text{irr}(S_d^{[2]})$. In the latter case one can take the set
\[ I=\{2(n^2+n+1): n\in \mathbb{Z}_{\geq 2}\}. \]
By \cite[Theorem 6.1.4]{Has} the hyper-K\"ahler fourfold $S_d^{[2]}$ is isomorphic to the Fano variety of lines on a smooth cubic fourfold. Of course, the degree of the image of the Fano variety of lines under the Pl\"ucker embedding is independent of the choice of the cubic fourfold and provides and upper bound on the degree of irrationality.
\end{question}

We will see in \S\ref{section:VariationalProperties} that the degree of irrationality of a K3 surface can only drop under specialization. Finally, we remark that there are several instances of $K3$ surfaces which have low degree of irrationality. For instance, the Enriques-Campedelli Theorem characterizes K3 surfaces with degree of irrationality $2$ as those containing a smooth hyperelliptic curve (see \cite{Dolg,Reid76} for modern references) and the first author \cite{Chen19} showed that a very general Kummer surface has degree of irrationality equal to $2$. Moretti and Rojas \cite{MR25} show that K3 surfaces of degree up to $14$ have degree of irrationality at most 4.
\item \textbf{Abelian surfaces:}\label{surfclass} For every abelian surface $A$, we have the inequalities
\begin{equation*}2=\textup{cov.gon}(A)< \textup{irr}(A)\leq 4.\end{equation*}
The upper bound follows from the fact that any Kummer surface has degree of irrationality equal to 2. Tokunaga and Yoshihara prove in \cite{TY} that the degree of irrationality of an abelian surface containing a smooth genus $3$ curve is $3$. In particular, very general $(1,2)$-polarized abelian surfaces have degree of irrationality $3$, as do some products of elliptic curves \cite{Y}. In \cite{M2}, the second author shows that the degree of irrationality of a very general $(1,d)$-polarized abelian surface is $4$ provided that $d$ does not divide $6$. Moretti then showed in \cite[Theorem C(2)]{Moretti23} that the degree of irrationality of a very general $(1,6)$-polarized abelian surface has degree of irrationality $3$, see \S \ref{polirr}. To the best of our knowledge, the degree of irrationality of very general $(1,1)$ and $(1,3)$ polarized abelian surfaces remains unknown. It is also unknown whether the degree of irrationality of a product of two very general elliptic curves is three or four.
\item \textbf{Enriques surfaces:} Enriques surfaces are branched double covers of the plane (\cite{E}, see also \cite{D}). Since they are not ruled, their covering gonality and degree of irrationality are both equal to $2$.
\item \textbf{Hyperelliptic surfaces:} Yoshihara has studied the degree of irrationality of hyperelliptic surfaces in \cite{Y4}. A hyperelliptic surface $S$ is a quotient variety of the form $(E\times F)/G$, where $E$ and $F$ are elliptic curves and $G$ is a subgroup of $F$ which acts on $F$ by translations and acts on $E$ arbitrarily. There are $7$ families of hyperelliptic surfaces and their degrees of irrationality are listed in Table \ref{hyperirr}. Note that if $j(E)=0$ or $1728$, then $E$ is respectively $\C/\mathbb{Z}[\omega]$ and $\C/\mathbb{Z}[i]$, where $\omega$ is a primitive third root of unity. In these cases, $\textup{Aut}(E)=\langle -\omega\rangle=\mathbb{Z}/6\mathbb{Z}$ and $\textup{Aut}(E)=\langle i\rangle=\mathbb{Z}/4\mathbb{Z}$ respectively. The degree of irrationality of the fourth type of hyperelliptic surface was only recently settled in \cite{mason25}, see \S\ref{polirr}.
\end{itemize}

\begin{center}
\begin{tabular}{|c|c|c|c|c|}
\hline
Order of $K_X$& $j(E)$ & $G$ & Action of $G$ on $E$ &$\textup{irr}(S)$\\
\hline
2 & Any & $\mathbb{Z}/2\mathbb{Z}$ & $e\mapsto -e$ & 2\\
2 & Any & $\mathbb{Z}/2\mathbb{Z}\times \mathbb{Z}/2\mathbb{Z}$ & $e\mapsto -e$, $e\mapsto e+c$, $c\in E[2]\setminus\{0_E\}$ & 2\\
3& 0 & $\mathbb{Z}/3\mathbb{Z}$ & $e\mapsto \omega e$ & 3\\
3 & 0& $\mathbb{Z}/3\mathbb{Z}\times \mathbb{Z}/3\mathbb{Z}$ & $e\mapsto \omega e$, $e\mapsto e+c$, $\omega c= c$ & 3\\
4 & 1728 & $\mathbb{Z}/4\mathbb{Z}$ & $e\mapsto ie$ & 3\\
4 & 1728 & $\mathbb{Z}/4\mathbb{Z}\times \mathbb{Z}/2\mathbb{Z}$ & $e\mapsto ie$, $e\mapsto e+c$, $ic=c$& 3\\
6 & 0 & $\mathbb{Z}/6\mathbb{Z}$ & $e\mapsto -\omega e$ & 3\\
\hline
\end{tabular}
\captionof{table}{Degree of irrationality of hyperelliptic surfaces}\label{hyperirr}
\end{center}

\subsection{\textbf{Kodaira dimension} $1$.}

Such surfaces are elliptic, i.e., they admit a fibration $S \to C$ whose general fiber is a smooth genus 1 curve (here we do not assume the existence of a section). We review what is known about measures of irrationality for elliptic surfaces regardless of Kodaira dimension. Clearly, if $S\to C$ is an elliptic surface then $\textup{cov.gon}(S)\leq 2$ and
\[ \textup{conn.gon}(S) \geq \textup{gon}(C). \]
Here, the \textit{connecting gonality} $\textup{conn.gon}(S)$ is the minimal gonality of a $2$-parameter family of curves on $S$ (the precise definition will appear in \S \ref{subsec:OtherMeasures}).

Yoshihara shows in Proposition 1 of \cite{Y} that if $S\to C$ has a section then
$$\textup{irr}(S)\leq 2\;\textup{gon}(C).$$
Of course, one can obtain more general results by considering the Jacobian fibration $J(S)\to C$ associated to an elliptic fibration $S\to C$. Recall that the index of an elliptic fibration $S\to C$ is the positive generator of the subgroup of $\Z$ generated by the degrees of intersections of curves on $S$ with fibers of the fibration. If $S$ admits an elliptic fibration of index $d$ there is a dominant rational map of degree $d^2$ (see \cite{BKL} p. 138):
$$S\dashrightarrow J(S).$$
Since $J(S)\to C$ is an elliptic fibration with a section, $\textup{irr}(J(S))\leq 2\;\textup{irr}(C)$ and
$$\textup{irr}(S)\leq 2d^2\; \textup{irr}(C).$$
Unfortunately, we know very little in terms of lower bounds for elliptic surfaces.

\subsection{Surfaces of general type}

As the topography of general type surfaces is too vast, we will just highlight a few interesting examples. Smooth surfaces in $\P^{3}$ of degree $\geq 5$ will be discussed in \S\ref{postech} whereas surfaces of maximal Albanese dimension will appear in \S\ref{abvar}. 

\begin{itemize}[wide, labelwidth=!, labelindent=0pt, label=$\diamond$] 
\item \textbf{Symmetric squares of curves:}

For any $k \geq 2$, the covering gonality of $\Sym^{k}C$ is bounded from above by $\gon(C)$ since one can cover $\Sym^{k}C$ by curves of the form $p_{1} + \cdots + p_{k-1} + C$.

Bastianelli proves in \cite{B} that if $g\geq 3$ the covering gonality of $\textup{Sym}^2 C$ coincides with the gonality of $C$. The covering gonality of the $k$-fold symmetric product of curves of dimension has also studied in \cite{BP25,BP24} for $k=2,3,4$.

\begin{problem}
Generalize the work of \cite{BP25} to $\text{Sym}^k C$ for $k\geq 5$.
\end{problem}

On the other hand, there are no conjectures about degree of irrationality in either direction which are expected to be optimal, at least for general curves. Symmetrizing any gonal map gives a dominant map $\Sym^{k}C \rightarrow \Sym^{k}\P^{1} \cong \P^{k}$, which implies that
\[ \irr(\Sym^{k}C) \leq \gon(C)^{k}. \]
In \cite{B}, Bastianelli studies measures of irrationality for the symmetric square of a smooth curve $C$ of genus $g$. In terms of upper bounds, he shows that 
$$\textup{irr}(\textup{Sym}^2 C)\leq \min\left\{\textup{gon}(C)^2, \ \frac{\delta_2(\delta_2-1)}{2}, \ \frac{(\delta_3-1)(\delta_3-2)}{2}-g\right\},$$
where $\delta_i$ is the minimal positive integer $d$ such that $C$ is birational to a non-degenerate curve of degree $d$ in $\mathbb{P}^i$.

Moreover, if $C$ is very general, then $\textup{irr}(\textup{Sym}^2 C)\geq g-1$. Bastianelli also proves that if $C$ is hyperelliptic then 
\begin{align*}\textup{irr}(\textup{Sym}^2 C)&\in \{3,4\} \qquad\text{ if }g\geq 2,\\
\textup{irr}(\textup{Sym}^2 C)&=4 \qquad\;\;\;\;\;\;\;\text{ if }g\geq 4.\end{align*}
In upcoming work, the second author will prove that $\text{irr}(\text{Sym}^2 C)\geq 4$ if $C$ is genus $3$ curve, and therefore that the degree of irrationality of the theta divisor of a genus $3$ hyperelliptic Jacobian is $4$.

\begin{question}
If $C$ is a general genus $3$ curve the symmetric product $\text{Sym}^2C$ is isomorphic to the theta divisor of $J(C)$. The Gauss map of this divisor is a dominant rational map of degree $6$ to $\mathbb{P}^2$. Is this a minimal degree dominant rational map to $\mathbb{P}^2$? If so, is it the unique such dominant rational map of degree $6$?
\end{question}

\item \textbf{Fano surface of a cubic 3-fold:} In \cite{GK19}, Gounelas and Kouvidakis study measures of irrationality of $S$, the Fano surface of lines on a smooth cubic $3$-fold $X$. They show the following inequalities: 
$$3\leq \textup{cov.gon}(S)\leq \textup{irr}(S)\leq 6.$$
Moreover, if $X$ is very general, then
\[ 4= \covgon(S)\leq 5=\conngon(S)\leq 6= \textup{irr}(S). \]
See the previous subsection for the definition of connecting gonality.

\begin{question}
Is there a smooth cubic threefold whose Fano surface of lines $S$ satisfies one of $\textup{cov.gon}(S)=3$, $\conngon(S)<5$, or $\textup{irr}(S)<6$? What can be said for higher-dimensional cubic hypersurfaces?
\end{question}

\end{itemize}

\vspace{0.5cm}


\section{Positivity techniques}
\label{postech}

There has been a significant amount of activity around the application of positivity techniques to measures of irrationality, with the focus being on hypersurfaces and complete intersections. Noether's theorem for the gonality of plane curves was first generalized to smooth surfaces in $\P^{3}$ by Lopez--Pirola \cite{LP} (for the covering gonality) and in the thesis of Cortini \cite{Cortini00} (for the degree of irrationality). Bastianelli--Cortini--De Poi \cite{BCD14} later revisited it for smooth surfaces and threefolds of large degree. A few years later, Bastianelli--De Poi--Ein--Lazarsfeld--Ullery \cite{BDELU17} computed the degree of irrationality of very general hypersurfaces in any dimension (and also gave bounds for the covering gonality of any smooth hypersurface). Smith extended some of these results to positive characteristic \cite{Smith20}. In a somewhat different direction, Chen--Stapleton \cite{CS20} established lower bounds for Fano hypersurfaces by degenerating to positive characteristic (following ideas of Koll\'{a}r \cite{Kollar95}). For complete intersections in projective space, it was conjectured in \cite{BDELU17} that measures of irrationality should behave multiplicatively in the degrees of the defining equations. This was in part inspired by a computation of Lazarsfeld for the gonality of smooth complete intersection curves \cite{Laz97}. Partial progress was made in this direction in a series of papers by Stapleton \cite{Stap}, Stapleton--Ullery, \cite{SU20}, Chen \cite{Chen24}, and Levinson--Stapleton--Ullery \cite{LSU23}. The conjecture was recently confirmed by Chen--Church--Zhao \cite{CCZ24}.

Let us begin by explaining how positivity considerations lead to lower bounds for the gonality of curves. This idea will turn out to be crucial in higher dimensions.

\begin{definition}
Let $X$ be a smooth projective variety, let $L$ be a line bundle on $X$, and let $U \subset X$ be a subset. We say that \textit{sections of $L$ separate $r$ distinct points of $U$} if the restriction map
\[ H^{0}(X, L) \rightarrow H^{0}(Z, L \big|_{Z}) \]
is surjective for any set $Z \subset U$ consisting of $r$ distinct points.
\end{definition}

Often times we will simply say that $L$ separates $r$ points in $U$. The following lemma provides lower bounds for the gonality of any curve (as an exercise, one can use this to establish Noether's theorem in Example~\ref{ex:noe}).

\begin{lemma}\label{lem:gonalityCurves}
    Let $C$ be a smooth projective curve and suppose $K_{C}$ separates $r$ points on some subset $C \setminus V$, where $V$ is a countable subset of closed points. Then $\gon(C) \geq r+1$.
\end{lemma}

\begin{remark}
In \cite{BDELU17}, the authors defined a slightly stronger notion of separating finite subschemes called $\mathrm{BVA}_{p}$, but it was pointed out in \cite{Stap} that the weaker notion still leads to lower bounds for gonality.
\end{remark}

Given a complete intersection $X \subset \P^{n+r}$ cut out by hypersurfaces of degrees $d_{1}, \ldots, d_{r}$, one can always project away from points on $X$ to obtain the upper bound
\[ \covgon(X) \leq \irr(X) \leq d_{1} \cdots d_{r} - r. \]
For lower bounds, a basic observation we will begin with is that the positivity of the canonical linear series plays an important role in studying measures of irrationality on hypersurfaces and complete intersections. One analogue of Lemma~\ref{lem:gonalityCurves} for higher-dimensional varieties is:

\begin{proposition}\label{prop:covgon}
    Let $X$ be a smooth projective variety. If $K_{X}$ separates $r$ points on $X \setminus V$, where $V$ is a countable union of proper subvarieties, then
    \[ \covgon(X) \geq r+1. \]
\end{proposition}

\begin{proof}
    This is a straightforward application of Riemann-Hurwitz to a covering family of curves $\cC \rightarrow X$ together with Lemma~\ref{lem:gonalityCurves}. See \cite{BDELU17} and \cite{Stap} for more details.
\end{proof}

\subsection{Hypersurfaces of large degree}

We will now survey the literature on measures of irrationality for hypersurfaces of large degree.

\begin{thm}[\cite{BCFS18, BDELU17}]
    Let $X = X_{d} \subset \P^{n+1}$ be a smooth hypersurface of degree $d \geq n+2$. Then $\covgon(X) \geq d-n$. If moreover $X$ is very general, then
    \[ \covgon(X) = d-\left\lfloor\frac{\sqrt{16 n+1}-1}{2}\right\rfloor\]
unless $n\in \{4\alpha^2+3\alpha,4\alpha^2+5\alpha+1|\alpha\in \mathbb{Z}_{>0}\}$, in which case the covering gonality may drop by one.
\end{thm}

\noindent The lower bound $\covgon(X) \geq d-n$ above follows immediately from Proposition~\ref{prop:covgon} and the adjunction formula for $K_{X}$.

The degree of irrationality can also be precisely described:

\begin{thm}[\cite{Cortini00, BCD14, BDELU17}]\label{thm:BDELU}
    Let $X = X_{d} \subset \P^{n+1}$ be a smooth hypersurface of degree $d \geq 2n+1$.
    \begin{enumerate}
    \item Suppose $n = 2$. Then $\irr(X) = d-1$ unless one of the following situations occurs:
    \begin{enumerate}[(a)]
    \item $X$ contains a line $\ell$ and a rational curve $R$ of degree $k$ which meets $\ell$ in $k-1$ points;
    \item $X$ contains a twisted cubic.
    \end{enumerate}
    In each of these cases, $\irr(X) = d-2$.
    \item Suppose $n = 3$. Then $\irr(X) = d-1$ unless one of the two following situations occurs:
    \begin{enumerate}[(a)]
        \item $X$ contains a non-degenerate rational scroll $S$ of degree $s$ and a line $\ell$ which is $(s-1)$-secant to $S$;
        \item $X$ contains a non-degenerate rational surface $S$ of degree $s$ and a line $\ell \subset S$ such that the residual intersection of $S$ with a general hyperplane $H$ containing $\ell$ is an irreducible rational curve $R$ which is $(s-2)$-secant line to $\ell$.
    \end{enumerate}
    In each of these two cases, $\irr(X) = d-2$.
    \item For any $n \geq 2$, if $X$ is very general then $\irr(X) = d-1$. Furthermore, if $\irr(X) \geq 2n+2$, then any map $X \dashrightarrow \P^{n}$ of degree $d-1$ is birationally equivalent to projection from a point on $X$.
    \end{enumerate}
\end{thm}

An important idea that appears in the above papers when studying the degree of irrationality for hypersurfaces is the \textit{Cayley--Bacharach condition}. The basic idea is as follows. Given a dominant map between $n$-dimensional varieties $f \colon X \dashrightarrow Y$, there is a trace map 
\[ \Tr \colon H^{0} (X, K_{X}) \longrightarrow H^{0}(Y, K_{Y}), \]
which was introduced by Mumford. Let $U\subset Y$ be the open locus where $f$ is well-defined and \'{e}tale. Given $y \in U$, the trace map takes the form
\[\Tr(\eta)(y) = \sum_{x \in f^{-1}(y)} \det(df^{-1})(\eta(x))\in \mybigwedge^n T_{U,y}^*, \]
where $df: T_{U,y}^*\longrightarrow  T_{f^{-1}(U),x}^*$ is the differential of $f$. This holomorphic form extends from $U$ to $X$ by Hartog's theorem.

If $H^{n,0}(Y) = 0$, which is to say that $Y$ has no holomorphic $n$-forms, then for a general $y \in Y$ we see that if $\eta$ vanishes on all but one of the points of $f^{-1}(y)$, then $\eta$ vanishes on the remaining point. This implies that $K_X$ cannot separate $\geq \deg f$ points on an open subset of $X$. The fibers of the map are said to satisfy the \textit{Cayley--Bacharach condition with respect to} $|K_{X}|$.

Now we will briefly sketch a proof of part (3). For hypersurfaces $X \subset \P^{n+1}$ of degree $d \geq 2n+1$, the key place where this Cayley-Bacharach property appears is to show that for maps $X \dashrightarrow \P^{n}$ of low degree, say $\delta < d$, the general fiber $X_{y} \coloneqq f^{-1}(y)$ over $y \in \P^{n}$ (viewed as a subvariety $X_{y} \subset \P^{n+1}$) must lie on a line $\ell_{y}$. Note that this is exactly what one would see if the map $f$ comes from projection from a point!

Recall from the covering gonality bound that $d-\delta \leq n$. Assuming $f$ is not given by projection from a point, in \cite{BDELU17} the authors observe that if one writes
\[ \ell_{y} \cdot X = X_{y} + F_{y}, \]
where $F_{y}$ is a zero-cycle of degree $d-\delta$, then the $F_{y}$ (or some subcycles) will sweep out a subvariety $S \subset X$ of dimension $s \geq 1$ having covering gonality $e \leq d - \delta$. Using the fact that a general point in $\P^{n+1}$ lies on exactly one of these lines $\ell_{y}$ (the family of lines forms what is called a \textit{first order congruence of lines}) and a numerical calculation involving the universal family of lines, they prove the inequality $e(n-s) \leq n$. But by some ideas of Ein \cite{Ein88} and Voisin \cite{Voisin96}, in a large enough degree range a very general hypersurface of degree $d$ does not contain an irreducible subvariety of dimension $s \geq 1$ with $\covgon(S) = e$. In particular, this contradicts the fact that $S$ has dimension $\geq 1$ to begin with.

Building on Theorem~\ref{thm:BDELU}, there are some natural questions that one can ask for hypersurfaces in higher dimensions:

\begin{question}
    For an arbitrary smooth hypersurface $X \subset \P^{n+1}$ of dimension $n \geq 4$ and sufficiently large degree, can one classify maps $X \dashrightarrow \P^{n}$ of degree $\leq d-1$? Likewise, is it true that any covering family which computes the covering gonality must come from a suitable family of singular plane curves?
\end{question}

Finally, beyond the case of surfaces and threefolds one can ask:

\begin{question}
    For a smooth hypersurface $X \subset \P^{n+1}$ of large degree, is it true that $\irr(X) \geq d-2$?
\end{question}

\subsection{Complete intersections of large degree}

Lazarsfeld \cite{Laz97} noticed early on that the gonality of a smooth complete intersection curve $C \subset \P^{r+1}$ of type $(d_{1}, \ldots, d_{r})$ is bounded from below by
\[ \gon(C) \geq (d_{1}-1) d_{2} \cdots d_{r} \]
(assuming the degrees are ordered as $d_{1} \leq d_{2} \leq \cdots \leq d_{r}$). In other words, the lower bound on gonality is \textit{multiplicative} in the degrees of the defining equations. The arguments were somewhat limited to the case of curves in that they relied on Bogomolov unstability of vector bundles on a surface containing $C$ and the fact that gonality is computed by a regular map rather than a rational map. This bound was later revisited in work of Hotchkiss--Lau--Ullery \cite{HLU20}, where they showed that the maps realizing the gonality of $C$ are computed as projections from suitable linear subspaces. For complete intersection varieties of higher dimension, the naive approach using Proposition~\ref{prop:covgon} together with the adjunction formula unfortunately only gives a lower bound that is \textit{additive} in the degrees $d_{i}$.

In a recent paper \cite{CCZ24}, the authors proved a significant generalization of Proposition~\ref{prop:covgon} which involves working with pairs and using Nadel vanishing (inspired by the work of Angehrn--Siu on Fujita's conjecture). Here, we state a simplified version which will be enough for our purposes:

\begin{thm}[{\cite[Theorem C]{CCZ24}}]
    \label{thm:CCZ1}Let $(X, H)$ be a smooth polarized variety. Suppose there exists an open subset $U \subseteq X$ and there exists a number $\alpha > 0$ such that any curve $C \subseteq X$ meeting $U$ satisfies $\deg_H{C} \ge \alpha$. Then there exists a constant $\delta \coloneqq \delta(X, H)$ such that the linear series $|K_X + d H|$ separates at least $(d - \delta \, \sqrt{d}) \cdot \alpha$ distinct points on $U$. Moreover, if $H$ is very ample, then one can take $\delta = 2\dim{X}$.
\end{thm}

Furthermore, the authors gave a lower bound for the degree of any curve on a general complete intersection variety of large enough degrees \cite[Theorem A]{CCZ24}. More precisely, if the complete intersection $X$ is cut out by polynomials of large enough degrees, then any curve on $X$ has degree bounded from below by the degree of $X$. Together with Theorem \ref{thm:CCZ1}, this was used to show:

\begin{thm}[{\cite[Theorem B]{CCZ24}}]
Let $X \subset \P^{n+r}$ be a general complete intersection variety of dimension $n$ cut out by polynomials of degrees $d_{1}, \ldots, d_{r} \ge n$. Then
\[ \covgon(X) \ge (d_1 - 2 (n+1)  \sqrt{d_1})(d_2 - n + 1) \cdots (d_r - n + 1) + 1. \]
\end{thm}

\noindent 
Asymptotically, this implies that for any $\epsilon > 0$ there is an integer $N(\epsilon ; n , r)$ such that when $d_i \ge N(\epsilon ; n, r)$ we have 
\[ \covgon(X) \ge (1 - \epsilon) d_1 \cdots d_r. \]
Somewhat unexpectedly, this covering gonality result applies to the \textit{general} complete intersection (as opposed to the \textit{very general} one), so in particular this holds for most complete intersections over $\Q$.

\subsection{Fano hypersurfaces}

In view of the techniques above, it is natural to ask about whether one can bound measures of irrationality for varieties with negative canonical bundle. It is well-known that Fano varieties are rationally connected \cite{KMM92}, so in particular their covering gonality is always equal to 1. For hypersurfaces, the Fano range coincides with $d \leq n+1$. In \cite{CS20}, the first author and Stapleton gave the first examples of Fano varieties with arbitrarily large degrees of irrationality:

\begin{thm}
Let $X_{n,d} \subset \P^{n+1}_{\C}$ be a very general hypersurface of dimension $n$ and degree $d$. If $d \geq n+1 - \sqrt{n+2}/4$, then
\[ \irr(X_{n,d}) \geq \sqrt{n+2}/4. \]
\end{thm}

\noindent Furthermore, the bound above actually holds for the minimal degree of a map to a \textit{ruled variety}. Work of Iskovskih and Manin shows that the degree of irrationality of any smooth quartic threefold is equal to 3 (see Example~\ref{ex:IskMan}), so the theorem above gives the first examples of rationally connected varieties $X$ with $\irr(X) \geq 4$.

The main idea behind the theorem above is to degenerate to positive characteristic in the spirit of Koll\'{a}r \cite{Kollar95}, who showed that certain $n$-dimensional cyclic covers $X$ of hypersurfaces in characteristic $p$ carry a large amount of differential $(n-1)$-forms. There are then two things to show: (1) maps to ruled varieties specialize, (2) the cyclic covers $X$ cannot admit low degree maps to ruled varieties. The criterion that allows the authors to prove part (2) is similar in spirit to the paragraphs following Theorem~\ref{thm:BDELU} and relies on the existence of a trace map for differential forms under separable maps.

Another class of varieties that are of significant interest are Calabi-Yau hypersurfaces. For example, Voisin \cite{Voisin04} has shown that a very general Calabi-Yau hypersurface $X_{d} \subset \P^{n+1}$ (where $d = n+2$) is not birationally covered by $r$-dimensional abelian varieties for any $r \geq 2$. The following question remains open:

\begin{question}
    Let $X \subset \P^{n+1}$ be a general Calabi-Yau hypersurface of dimension $n \geq 3$. Is $X$ covered by elliptic curves?
\end{question}

In this direction, Voisin has made a number of interesting observations. For example, a potential covering family of elliptic curves cannot be constant in moduli. As a byproduct, she shows that such a hypersurface which is covered by elliptic curves must contain a uniruled divisor. On the other hand, for the very general quintic threefold ($n = 3$) this violates Clemens' conjecture on the finiteness of rational curves of fixed degree.

\vspace{0.5cm}

\section{Abelian varieties}\label{abvar}

Since abelian varieties have trivial cotangent and canonical bundles, one cannot use positivity techniques to obtain lower bounds on their measures of irrationality. Nonetheless, several results point to the fact that abelian varieties of large dimension, and more generally varieties with large Albanese dimension are far from rational.

The first such result applies more generally to varieties with large holomorphic length. It is one of the few results which provide interesting lower bounds on the degree of irrationality of a broad class of algebraic varieties. Introduced in \cite{AP2}, the holomorphic length $\text{hol.length}(X)$ of a smooth projective variety $X$ is the largest non-negative integer $r$ such that there exists holomorphic forms $\omega_1, \ldots, \omega_r \in \bigoplus_{i=1}^{\dim X} H^0(X,\Omega^i)$ with $\omega_1\wedge \dots\wedge \omega_r \neq 0$. 

Note that the holomorphic length of a variety is at most its dimension and greater or equal to its Albanese dimension. Alzati--Pirola showed that holomorphic length can be used to give lower bounds for the degree of irrationality:

\begin{thm}[\cite{AP2}]\label{APthm}
Let $X$ be a smooth projective variety.  The degree of irrationality of $X$ is greater than its holomorphic length.
\end{thm}

\begin{corollary}
    The degree of irrationality of $X$ is greater than its Albanese dimension. In particular, the degree of irrationality of an abelian variety is greater than it dimension.
\end{corollary}

\begin{example}
    Let $X$ be a hyper-K\"ahler $2n$-fold with holomorphic symplectic form $\omega$. Since $\omega$ is non-degenerate, $\omega^n\neq 0$ and therefore 
    $$\irr(X)>\text{hol.length}(X)=n.$$
\end{example}

In many cases, the bound in Theorem~\ref{APthm} is still the best known bound (see the case of abelian surfaces in \S\ref{surfclass}). In the hyper-K\"ahler setting, Voisin has asked if the Alzati--Pirola bound can be improved as in \cite{M2} to $\text{irr}(X)\geq 4$ for some (or all) hyper-K\"ahler fourfolds \cite[Question 1.10]{Voi22}.

In a different direction, lower bounds for the degree of irrationality of subvarieties of very general abelian varieties have been obtained using a specialization technique due to Pirola. The first example of such a result is a theorem of Pirola from \cite{Pirola} stating that a very general abelian variety of dimension at least $3$ does not contain hyperelliptic curves. It is worth noting that given a curve of gonality $k$ in an abelian variety, by translating one can produce an isotrivial covering family of curves of gonality $k$. Thus, the covering gonality of an abelian variety $X$ coincides with the minimal gonality of any curve in $X$.

We will now briefly explain the idea for how to bound the degree of irrationality of subvarieties of an abelian variety. Given an abelian $n$-fold $X$, suppose $Z\subset X$ is a $d$-dimensional subvariety with a degree $k$ dominant rational map $\varphi \colon Z \dashrightarrow \mathbb{P}^d$. Then the closure of the set of fibers of $\varphi$ is a $d$-dimensional rational subvariety of $\Sym^k X$. Let $\X \rightarrow T$
be a locally complete family of $\underline{d}=(d_1,\dots, d_n)$-polarized abelian $n$-folds, which is to say that the classifying morphism $T \rightarrow \mathcal{A}_{\underline{d}}$ to the moduli space of $\underline{d}$-polarized abelian varieties is dominant. In order to show that such a $Z$ does not exist on a very general member of $\X \rightarrow T$, it suffices to show that if $k$ is small then the relative symmetric product
\[ \Sym^k_T \X \coloneqq \X \times_{T} \cdots \times_{T} \X / \mathfrak{S}_{k} \]
does not contain an irreducible subvariety such that the map to $T$ is dominant, has relative dimension $d$, and has rational generic fiber.

Assuming for contradiction that $\mathcal{R}\subset \Sym^k_T \mathcal{X}$ is such a component, we can then specialize to a locus $T' \subset T$ parametrizing abelian varieties $X$ isogenous to a product $X' \times E$, where $E$ is an elliptic curve (varying in moduli) and $X'$ is a fixed abelian $(n-1)$-fold. If we write $\X_{T'} \coloneqq \X \times_{T'} T$ and $R_{T'} \coloneqq R \times_{T'} T$, then the composition $X\rightarrow X'$ of the isogeny with the projection onto the first factor induces a map
\[ \Sym^k_{T'}(\X_{T'}) \rightarrow \Sym^k X'. \]
Since $R \times_{T'} T \subset \Sym^k_{T'} (\X_{T'})$ is a one-parameter family of rational $d$-folds, its image in $\text{Sym}^k X'$ is in particular uniruled and the expectation is that it has dimension at least $(n+1)$.

This sets up an inductive argument where we specialize to abelian varieties splitting off an elliptic isogeny factor and project to the symmetric product of the complementary factor, at each step obtaining a larger and larger uniruled subvariety of a symmetric product of an abelian variety. However, some assumptions are necessary to argue that the dimension does grow at each step and checking that these assumptions are always satisfied is tricky. Voisin carried out a similar inductive argument in \cite{V} to show that the minimal gonality of a curve in a very general abelian variety of dimension at least $2^{k-2}(2k-1) + (2^{k-2}
-1)(k-2)$ is at least $k+1$.

The inductive argument outlined was carried out for $d=1$ in \cite{M1} to show that a very general abelian variety of dimension at least $2k-4$ does not contain curves of gonality $\leq k$ when $k\geq 4$, thereby proving a conjecture formulated in \cite{V}. Using some results from the theory of generic vanishing, a strengthening of the inductive argument was completed in \cite{CMNP} to prove the following:
\begin{thm}\label{CMNP}
    Let $A$ be a very general abelian variety of dimension at least $3$. The degree of irrationality of a $d$-dimensional subvariety of $A$ is at least $d+(\dim A+1)/2$. In particular, the degree of irrationality of $A$ is at least $(3\dim A+1)/2$.
\end{thm}

Note that this is an improvement of the Alzati--Pirola bound by $(\dim A-1)/2$, although it only applies to subvarieties of very general abelian varieties. The theorem above remains the best result for very general abelian varieties which applies regardless of the degree of the polarization. In his thesis, the second named author generalizes the results in \cite{M2} (see \S\ref{kodsection}) and proves that for each $n$, there are integers $M_n$ such that if $d_n\nmid M_n$ then a very general abelian $n$-fold with polarization type $(d_1=1, d_{2},\ldots, d_n)$  has degree of irrationality at least $2n$.

\begin{question}
    For a very general polarized abelian variety of dimension $n \geq 3$, does $\irr(A)$ grow with the degree of the minimal polarization?
\end{question}

\begin{problem}
    Give interesting constructions of dominant rational maps of low degree from some abelian $n$-fold to $\P^{n}$. For instance, Jacobians of curves or abelian varieties with extra automorphisms are a natural starting point.
\end{problem}

The specialization technique discussed above also yields cycle-theoretic results of independent interest. For a smooth projective variety $X$ one can consider the map
$$\text{Sym}^k X\longrightarrow \text{CH}_0(X) $$
which associates to a $k$-tuple of points on $X$ the associated effective zero-cycle of degree $k$. Fibers of this map provide fine geometric information about rational equivalence of zero-cycles on $X$. For example, when $X$ is a K3 surface and $k=1$, an integral curve in $X$ contracted by this map is called a constant cycle curve and was studied in \cite{Huy}. See \cite[Theorem 1.1]{CMNP} for a cycle-theoretic analogue of Theorem \ref{CMNP}.

The lower bounds on the covering gonality and the degree of irrationality discussed in this section are not expected to be sharp. On the other hand, it is quite difficult to produce curves that are not complete intersections on general abelian varieties of large dimension since they do not arise as Jacobians or ramified Pryms. As a result, there is an enormous gap between known lower bounds and upper bounds; where the actual asymptotic for $n$-dimensional abelian varieties lies as $n \rightarrow \infty$ remains a complete mystery, and one is in a similar situation when it comes to the degree of irrationality.

\begin{question}
Does the minimal genus of a curve on a very general abelian $n$-fold grow slower than $C \cdot n!$ for some positive constant $C$? What can one say about the rate of growth of the degree of irrationality of a very general abelian variety of dimension $g$?
\end{question}

\vspace{0.5cm}

\section{Open questions and complements}
\label{open}

In this section, we would like to revisit some of the questions that appeared in \cite{BDELU17} and raise a number of new problems.

\subsection{Products of curves}
\label{prodsection}
Consider a product of two curves $S = C_{1} \times C_{2}$. As mentioned in Example~\ref{ex:productcurves}, the covering gonality is the minimum of $\gon(C_{1})$ and $\gon(C_{2})$, and can be computed by the trivial family coming from one of the projections. Moreover, one has the following bounds for the degree of irrationality (see Example \ref{ex:productcurves}):
\[ \max \{ \gon(C_{1}), \gon(C_{2}) \} \leq \irr(C_{1} \times C_{2}) \leq \gon(C_{1}) \cdot \gon(C_{2}). \]
Furthermore, the methods in \cite{CM23} lead to lower bounds of the form $\gtrsim \gon(C_{1}) + \gon(C_{2})$.

In \cite{CM23}, it was shown that any product of hyperelliptic curves (of genus $\geq 2$) has degree of irrationality equal to 4. There are also some partial results about curves of higher gonality. For instance, the authors showed that the upper bound on $\irr(C_{1} \times C_{2})$ is achieved provided that the gonalities of $C_{1}$ and $C_{2}$ are small relative to their genera. This forces both curves to be quite special. However, we expect the upper bound to be achieved when $C_{1}$ and $C_{2}$ are (very) general in moduli.

\begin{conjecture}
    Let $C_{1}, C_{2}$ be curves of genera $\geq 2$ which are very general in moduli. Then $\irr(C_{1} \times C_{2}) = \gon(C_{1}) \cdot \gon(C_{2})$.
\end{conjecture}

This expectation stems from the fact that the space of differential 2-forms has dimension $g_{1} g_{2}$. Thus, we would expect that there are enough forms to separate the fibers of a map $\varphi \colon C_{1} \times C_{2} \dashrightarrow \P^{2}$ if $\deg \varphi \leq \gon(C_{1}) \cdot \gon(C_{2}) \approx \frac{g_{1}g_{2}}{4}$. One approach towards the conjecture that would make the above heuristic precise is to argue that the fibers of such a map $\varphi$ must form a pattern that closely resembles ``grid points" inside $C_{1} \times C_{2}$, and that sections of $K_{C_{1} \times C_{2}}$ can be made to interpolate through all but one of the points.

Another approach to the conjecture above would be to show that any curve $C'$ which dominates both $C_{1}$ and $C_{2}$ has gonality at least $\gon(C_{1}) \cdot \gon(C_{2})$. A curve $C'$ that maps to both $C_{1}$ and $C_{2}$ is called a \textit{correspondence}. In the context of measures of irrationality, correspondences have appeared in work of Lazarsfeld and the second author \cite{LM23,LMII} (see \S\ref{measass}). At the moment, there are no sharp lower bounds for the gonality of a correspondence between two curves that are (very) general in moduli.

It is also interesting to ask what happens when one takes products of more copies of curves:

\begin{question}
Given $n$ very general hyperelliptic curves $C_1,\dots, C_n$, what is the degree of irrationality of their product $C_{1} \times \cdots \times C_{n}$? Does it grow sub-exponentially in $n$?

\end{question}

\subsection{Polarized degree of irrationality}\label{polirr}

One of the inherent difficulties in computing the degree of irrationality of a smooth projective variety $X$ is that rational maps $X\dashrightarrow \mathbb{P}^n$ are given by the choice of a line bundle $L$ on $X$ and an $(n+1)$-dimensional subspace of $H^0(L)$. Even if $\text{Pic}^0(X)$ is discrete, there is a countable union of Grassmanians parametrizing such rational maps and we are unable to bound the linear systems computing measures of irrationality. For instance, Lazarsfeld asks:

\begin{question}[Lazarsfeld]
    Given a surface $S$ with $\Pic(S) = \Z$, can one bound the integer $m$ for which there exists a sublinear series $|V| \subset |mL|$ computing $\irr(S)$? What about bounding the integer $m$ for which $\covgon(S)$ is realized by a family of curves algebraically equivalent to $mL$?
\end{question}

To simplify the situation, Moretti \cite{Moretti23} fixes the line bundle $L$, i.e., he considers a polarized variety $(X, L)$, and studies maps given by sub-linear series of $|L|$:
\[ X \dashrightarrow \P (H^{0}(L)^{\vee}) \dashrightarrow \P^{n}. \]
These give rise to an interesting vector bundle construction analogous to the one of Lazarsfeld--Mukai bundles, which were used by Lazarsfeld to study linear series on smooth curves contained in K3 surfaces \cite{Lazarsfeld86}. Given a variety $X$ and a linear series $|V|$ corresponding to a rational map $\varphi_{V} \colon X \dashrightarrow \P^{n}$, one may consider the kernel bundle, which sits in the following exact sequence
\[ 0 \rightarrow E^{\vee} \rightarrow V \otimes \O_{S} \rightarrow L. \]
Note that the cokernel of the map on the right is supported on the base-locus of $\varphi_{V}$. Since $H^0(L^\vee)=0$, we can view $\P(V^{\vee})$ as a subspace of $\P(H^{0}(E))$, and the fibers of $\varphi_{V}$ can be viewed as the zero loci $Z(s)$ of sections $s$ of $E$; over $[s] \in \P(V^{\vee}) = \P^{n}$, the fiber is given by $Z(s)$ (excluding the base locus).

Moretti's main observation is that, under some additional assumptions on $E$, this process can be reversed to construct rational maps to $\P^{n}$. Namely, given a sufficiently general vector bundle $E$ of rank $n$ and at least $n+1$ linearly independent sections generating $E$ in codimension 2, one can construct a dominant rational map $X\dashrightarrow \mathbb{P}^n$ of degree $\leq c_{n}(E)$. Furthermore, this can be improved by imposing base points; for $V^{\vee} \in \Gr(n+1, H^{0}(E))$, the base locus $\Bl(V^{\vee})$ is the union of the locus where the sections do not generate $E$ and of the locus where $E$ is not locally free, and we let $\text{bl}(V^{\vee}) \coloneqq \deg(Z(s) \cap \Bl(V^{\vee}))$, where $s \in V^{\vee}$ is a general section. With this notation, one obtains (see \cite[Proposition 1.5]{Moretti23})
\[ \deg(\varphi_{V}) = c_{n}(E) - \text{bl}(V^{\vee}). \]

\begin{example}[Moretti]
Given $(X, L)$ a K3 surface of genus 6 with $\rho(X) = 1$, we will show that $\irr(X) = 3$. Since $L$ is a genus $6$ polarization, it satisfies $L^{2} = 10$ and $h^{0}(X, L) = 7$. By work of Mukai \cite{Mukai87}, one can find a unique stable rank 2 vector bundle $E$ of minimal degree $c_{2}(E)$ with the following invariants:
\[ c_{1}(E) = L, \quad c_{2}(E) = 4, \quad h^{0}(E) = 5. \]
For any point $P \in X$, the kernel $V_P^\vee$ of the evaluation map $H^0(E)\longrightarrow H^0(E\otimes \mathcal{O}_{X,p})$ has dimension 3. Note that a generic section of $V_{P}^{\vee}$ vanishes at $P$ with order $1$, so that
\[ \deg(\varphi_{V_{P}}) = c_{2}(E) - 1 = 3. \]
On the other hand, $\irr(X) = 2$ would imply that $X$ admits a rational anti-symplectic involution. For a K3 surface $X$, it is straightforward to show that $\Bir(X) = \Aut(X)$ so one would then conclude that $\Z/2\Z \leq \Aut(X)$. However, it is well-known that this only occurs for a K3 surface $X$ with Picard rank 1 if $X$ is a double-cover of $\P^{2}$ \cite[pg. 342, Corollary 2.12]{Huybrechts16}. Therefore, $\irr(X) = 3$.
\end{example}

In a recent preprint, Mason \cite{mason25} has adapted some of these methods to study hyperelliptic surfaces. He used this to complete the remaining case in the classification of hyperelliptic surfaces by constructing degree three maps from certain hyperelliptic surfaces to $\P^{2}$.

\subsection{Variational properties}\label{section:VariationalProperties}

For a smooth family of varieties, one can ask how the degree of irrationality of the fibers varies. First of all, for curves it is well-known that the gonality cannot increase under specialization (this holds even in flat families). Therefore, in smooth families of varieties the covering gonality cannot increase under specialization (see e.g. \cite[Proposition 2.2]{GK19} or \cite[Proposition 2.7]{Chen24}). On the other hand, it is still an open problem as to whether this is always the case for the degree of irrationality.

\begin{question}\label{prob:inequalitydegirr}
    Over $\C$, let $\X \rightarrow B$ be a smooth proper morphism over a pointed smooth connected curve $(B,0)$. If the very general fiber $\X_{b}$ satisfies $\irr(\X_{b}) \leq d$, then is it true that $\irr(\X_{0}) \leq d$?
\end{question}

\noindent When $d = 1$, the statement of Question~\ref{prob:inequalitydegirr} follows from the work of Kontsevich and Tschinkel \cite{KT}, who showed that rationality specializes even in mildly singular families. For higher values of $d$, even for smooth families of algebraic surfaces surprisingly little is known. In \cite[Proposition C]{CS20}, the first author and Stapleton used a specialization statement about maps to ruled varieties to show that Question~\ref{prob:inequalitydegirr} is true for special families such as simply connected surfaces or strict Calabi-Yau threefolds. In fact, for surfaces the argument works more generally when the central fiber does not admit any maps to a curve of genus $g \geq 1$. In particular, this shows that the degree of irrationality in a family of K3 surfaces cannot increase under specialization.

Under an additional assumption, which is that the degree of irrationality on the very general fiber can be computed as a rational group quotient, the first author and Esser showed that the degree of irrationality is lower-semicontinuous for families of integral klt surfaces \cite{CE24}. However, in general the authors do not expect the degree of irrationality to be lower-semicontinuous in smooth families. For instance, a negative answer to Question~\ref{prob:inequalitydegirr} would follow from the existence of a pair of elliptic curves $E, E'$ such that $\irr(E \times E') = 4$. Indeed, one could then take a one-parameter family of $(1,2)$-polarized abelian surfaces specializing to $E\times E'$, such that the general fiber has degree of irrationality $3$ \cite{Y}.

\subsection{Other measures of irrationality}\label{subsec:OtherMeasures}

There is another measure of irrationality that one can define which mimics the notion of rational connectedness:

\begin{definition}
The \textit{connecting gonality} of $X$ is
\[ \conngon(X) \coloneqq \min\begin{cases} c \in \mathbb{Z}_{\geq 0}\;\Bigg|\begin{rcases} \forall p, q \in X \text{ two general points}, \ \exists \text{ an} \\ \text{integral curve $C \ni p, q$ of gonality } c \end{rcases}.\end{cases} \]
\end{definition}

As far as the connecting gonality is concerned, we have the following result for smooth surfaces $X\subset \mathbb{P}^{3}$ of degree $d \geq 4$. For any $x\in X$, the intersection $C_x \coloneqq T_x\cap X$ of $X$ with the embedded tangent plane to $X$ at $x$ is an irreducible plane curve of degree $d$ with a double point at $x$. Lopez and Pirola show in \cite{LP} that varying $x$ yields a covering family of minimal gonality and that it is unique for a general surface $X$. Note that the projective dual of $X$ is a surface $\check{X}\subset \check{\mathbb{P}}^3$. Given two very general points $y,y'\in X$, there is a line $l\subset \check{\mathbb{P}}^3$ parametrizing planes containing $y$ and $y'$. Since this line intersects $\check{X}$, two very general points of $X$ can be connected by a curve of the form $C_{x}$. It follows that
\[ \textup{conn.gon}(X)=\textup{cov.gon}(X)=d-2. \]

For self-products of a variety, Ellenberg has asked about the following:

\begin{question}[Ellenberg]
    Given a variety $X$, we can define $f_{X}(k) \coloneqq \sqrt[k]{\irr(X^{k})}$. Note that
    \[ f_{X}(k) \leq \irr(X) \qquad \forall k\in \mathbb{Z}_{>0}, \]
    and $f_{X}(k) \geq f_{X}(rk)$ for all $r \in \Z_{> 0}$. Is $f_{X}$ non-increasing? What other properties does the function $f_{X}$ have?
\end{question}

\begin{question}
Are there any higher-dimensional varieties $X$ for which degree of irrationality of $\Sym^{k}X$ grows exponentially in $k$?
\end{question}

For a curve $C$ of genus $g$, the question above is false since $\Sym^{k}C$ is birational to a product of $\Jac(C)$ with projective space for $k \geq g$. If $X$ is any variety with even-degree holomorphic forms, then interestingly one can give linear lower bounds in $k$ for the degree of irrationality of $\Sym^{k}(X)$ using holomorphic length (see \S3).

Bastianelli has defined two other notions that mimic unirationality and stable irrationality.

\begin{definition}
The \textit{degree of uni-irrationality} of $X$ is defined as
    \[ \uniirr (X) \coloneqq \min\{ \irr(T) \mid \exists \ T \dashrightarrow X \text{ dominant} \}. \]
\end{definition}

\begin{definition}
The \textit{stable degree of irrationality} of $X$ is defined as
\[  \stabirr(X) \coloneqq \min\{ \irr(X \times \P^{m}) \mid m \in \Z_{\geq 0} \}. \]
\end{definition}

\noindent Note that $\irr(X \times \P^{m+1}) \leq \irr(X \times \P^{m})$ so the stable degree of irrationality can also be defined as a limit.

Bastianelli \cite{Bastianelli17} computed these invariants for smooth surfaces $S \subset \P^{3}$ of degree $\geq 5$. Interestingly, $\uniirr(S)= d-2$ iff $S$ contains a rational curve (and otherwise it is equal to $d-1$). Shortly after that, Yang \cite{Yang} showed that these invariants coincide with the degree of irrationality for very general hypersurfaces $X \subset \P^{n+1}$ of degree $d\geq 2n+2$.

\subsection{Measures of association}\label{measass}
Measures of irrationality aim to quantify how far a variety is from being birational to $\mathbb{P}^n$, but one can ask how far two varieties $X$ and $Y$ are from being birational. This is the goal of measures of associations, birational invariants of pairs of varieties $(X,Y)$ which were first studied in two articles of Lazarsfeld and the second author \cite{LM23,LMII}. Typically, varieties $X$ and $Y$ of the same dimension will not admit dominant rational maps between each other, so one cannot consider the minimal degree of such maps as we did to define the degree of irrationality. However, there is always an irreducible subvariety of $Z\subset X\times Y$ mapping generically finitely to $X$ and $Y$.

We define the \emph{correspondence degree} of the pair $(X,Y)$ as
$$\corrdeg(X,Y)=\text{min}_{Z\subset X\times Y} (\deg(Z/X)\cdot\deg(Z/Y)),$$
where the minimum is taken over all irreducible subvarieties of $X\times Y$ mapping generically finitely to each of the factors. The logarithm of this invariant gives a metric on the set of birational types of $n$-folds. This invariant is related to the degree of irrationality of $X$ and $Y$ by the inequality
$$\corrdeg(X,Y)\leq \text{irr}(X)\cdot\text{irr}(Y),$$
which is obtained by choosing for $Z$ an appropriate irreducible component of the (birational) fiber product of dominant rational maps to $\mathbb{P}^n$.

Similarly, one defines the joint covering gonality of $X$ and $Y$ to be
$$\covgon(X,Y)=\text{min}_{Z\subset X\times Y}\covgon(Z),$$
where again the minimum is taken over all irreducible subvarieties of $X\times Y$ mapping generically finitely to each of the factors. A similar argument as in the case of the correspondence degree gives
$$\max(\covgon(X),\covgon(Y))\leq \covgon(X,Y)\leq \covgon(X)\cdot \covgon(Y).$$
Our intuition is that for most pairs of varieties $(X,Y)$ the upper bounds above should be attained, which is to say that $X$ and $Y$ should be ``maximally birationally independent". It is not so clear how to make this intuition precise, and we believe there are many ways to achieve this.

To compare varieties $X$ and $Y$ of dimensions $m$ and $n$ respectively, where $m\leq n$, it is natural to consider the \emph{stable correspondence degree}:
$$\text{stab.corr}(X,Y)=\text{min}_{r\geq n}\corrdeg(X\times \mathbb{P}^{r-m},Y\times\mathbb{P}^{r-n}).$$
This invariant was introduced and studied in \cite {Pas}, where some of the results of \cite{LM23} were generalized to the stable correspondence degree.

\subsection{Arithmetic measures of irrationality}

Finally, we would like to touch upon the connection between measures of irrationality and some more arithmetic notions. Throughout this section, let $X$ denote a smooth, projective, and geometrically integral variety over a number field $k$. The degree of irrationality $\irr(X/k)$ still makes sense when the map is defined over $k$, and the same for the covering gonality.

In the arithmetic world, one can also define various notions which are tied to the existence of some (or many) degree $d$ points, namely points whose residue field is a degree $d$ algebraic extension over $k$. This leads to

\begin{definition}
    \begin{enumerate}
    \item The density degree set $\delta(X/k)$ is the set of all positive integers $d$ such that the set of degree $d$ points on $X$ is Zariski dense.
    \item The minimum density degree $\min(\delta(X/k))$ is the smallest positive integer in $\delta(X/k)$.
    \end{enumerate}
\end{definition}

Note that by considering the point together with its Galois conjugates, a degree $d$ point on $X$ gives rise to a $k$-rational point on $\Sym^{d}(X)$. We refer the reader to the excellent survey of Viray--Vogt \cite{VV24} for a more in-depth introduction to these ideas.

For a curve $C/k$, the main strategy for understanding dense sets of degree $d$ points has been to pass to the $d$-th symmetric product where they become rational points. Either infinitely many of these rational points are contained in a fiber of the Abel--Jacobi map, in which case $C$ has a \textit{geometric} explanation for them (some sort of linear series), or the image of the Abel--Jacobi map contains infinitely many rational points. In the latter case, one can then apply Faltings' theorem. This idea has appeared in the work of many authors including Abramovich \cite{Abramovich91}, Abramovich--Harris \cite{AH91}, Debarre--Fahlaoui \cite{DF93}, and Harris--Silverman \cite{HS1991}. The analogue of Question~\ref{Q:irr=densitydeg} for smooth plane curves was answered by Debarre--Klassen \cite{DK94}, and this was extended to smooth ample curves on some other surfaces by Smith--Vogt \cite{SmithVogt22}.

Before considering higher-dimensional varieties, we remark that one can define ``potential" analogues of the density degree set and minimum density degree, which take into account possible finite extensions $k'/k$. We will not attempt to define these definitions precisely since they will not appear later in this survey. However, the relevance is that a map defined over the algebraic closure $\bar{k}$ is always defined over a finite field extension, so many of the existing constructions of maps that compute the degree of irrationality lead to elements in the potential density degree set. In fact, to produce a dense set of degree $d$ points a correspondence is often enough:

\begin{example}
If $\irr(X/k) = d$, then a consequence of Hilbert irreducibility (see \cite[Proposition 3.3.1]{VV24}) is that $d \in \delta(X/k)$. More generally, suppose there exists a nontrivial correspondence
\begin{center}
\begin{tikzcd}[column sep=small]
& Z \arrow[ld, swap, "f"] \arrow[rd, "g"] & \\
\P^{n} & & X
\end{tikzcd}
\end{center}
where $f, g$ are generically finite (and do not factor non-trivially through a common map) and $f$ has degree $d$. We claim that $d \in \delta(X/k)$. To see this, note that the pre-image under $f$ of a general line $\ell$ defined over $k$ is a smooth irreducible curve $C$ with $d \in \delta(C/k)$ by Hilbert irreducibility. Since the map $C \rightarrow g(C)$ is birational, varying $\ell$ gives a family of curves which cover $X$ and have $d$ in their density degree set. Therefore, $d \in \delta(X/k)$.
\end{example}

\begin{question}\label{Q:irr=densitydeg}
Consider a sufficiently large positive integer $d$ and $X\subset \mathbb{P}^3_\Q$ a smooth projective surface of degree $d$. Is it true that $\text{min }\delta(X/\mathbb{Q})=\text{irr}(X/\mathbb{Q})?$
\end{question}
One expects most degree $d$ points on high degree surfaces in $\mathbb{P}^3$ to be supported on lines defined over $\mathbb{Q}$. In \cite{HKM} Huang, Kadets, and the second author prove such a statement for the universal hypersurfaces in $\mathbb{P}^{n+1}$. Of course, one can always look at smooth surfaces containing a rational curves defined over $\mathbb{Q}$, which forces us to either ask about density of non-collinear points or density in moduli of surfaces admitting non-collinear points.
\begin{question} 
Consider a sufficiently large positive integer $d$ and $X\subset \mathbb{P}^3_\Q$ a smooth projective surface of degree $d$. Can degree $d$ points on $X$ which are not contained in a line be dense in $X$? Is the subset of $|\mathcal{O}_{\mathbb{P}^3_\Q}(d)|$ consisting of surfaces defined over $\mathbb{Q}$ containing a degree $d$ point not lying in a line dense?
\end{question}

One case we would like to mention is that of products of curves, where there has been some recent progress towards understanding their density degree set \cite{BFGPRV25}. In a slightly different direction, there has been a great deal of interest in understanding rational and quadratic points on hyperelliptic curves. It is expected that most do not have any unexpected rational or quadratic points (cf. \cite{Bhargava2013, PS14, Granville07, LS24} and the references therein). In light of this as well as the fact that the degree of irrationality of any product of hyperelliptic curves is equal to 4, it is tempting to conjecture the following for the density degree set:

\begin{conjecture}
    Most pairs of hyperelliptic curves $(C, D)$ of sufficiently large genera satisfy
    \[ \min \delta(C \times D/\Q) = 4. \]
\end{conjecture}

\noindent Note that one can reduce this to the Bombieri--Lang conjecture about rational points on the quotient of $C \times D$ by the diagonal involution.

Another potential direction of future research would be to understand whether a dense set of degree $d$ points in $X$ must come from a geometric source such as a correspondence or the existence of low gonality curves in $X$. Already for quadratic points on abelian varieties, the following is unknown:

\begin{question}\label{ques:A/Q}
Let $A$ be an abelian variety defined over $\mathbb{Q}$ which has Mordell--Weil rank $0$. Suppose that $2\in \delta(A/\mathbb{Q})$. Does $A$ contain a hyperelliptic curve defined over $\mathbb{Q}$?
\end{question}

As Kadets pointed out to us, it is expected that $2 \in \delta(A/\Q)$ for any abelian variety $A$ defined over $\Q$. In particular, Question~\ref{ques:A/Q} is closely related to:

\begin{question}
Does every abelian variety defined over $\Q$ contain a hyperelliptic curve?
\end{question}

\noindent Interestingly, the geometric methods from \S3 do not seem to apply. 

\noindent \textbf{Acknowledgements.} We would like to thank AIM for its hospitality and for providing a stimulating environment during the March 2024 workshop ``Degree $d$ points on algebraic surfaces".

We have learned about this subject from a wide variety of mentors, friends, and collaborators. In particular, we would like to thank Rob Lazarsfeld, Claire Voisin, Gian Pietro Pirola, David Stapleton, Sam Grushevsky, John Christian Ottem, Frank Gounelas, Bianca Viray, Benjamin Church, Louis Esser, Lena Ji, Borys Kadets, Federico Moretti, Giovanni Passeri, and Junyan Zhao.


\bibliographystyle{alpha}
\bibliography{biblio.bib}

\footnotesize{
\textsc{Department of Mathematics, Harvard University, Cambridge MA 02138} \\
\indent \textit{E-mail address:} \href{mailto:nathanchen@math.harvard.edu}{nathanchen@math.harvard.edu}

\textsc{Instituto de Matem\'{a}tica Pura e Aplicada, Rio de Janeiro, RJ 22460-320} \\
\indent \textit{E-mail address:} \href{mailto:olivier.martin@impa.br}{olivier.martin@impa.br}
}

\end{document}